\documentclass[a4paper,12pt]{amsart}
\usepackage{latexsym}
\usepackage{a4wide}
\usepackage{amscd}
\usepackage{graphics}
\usepackage{graphicx}
\usepackage[english]{varioref}
\usepackage{amsmath}
\usepackage{amssymb}
\usepackage{mathrsfs}
\usepackage{amsthm}
\usepackage{amssymb,tikz}
\usepackage{verbatim}
\usepackage{stmaryrd}
\usepackage[xindy]{glossaries}

\usepackage{tikz}
\tikzstyle{mybox} = [draw=black, very thick, rectangle, rounded corners, inner ysep=5pt, inner xsep=5pt]

\usepackage{bbm}
\usepackage{soul}
\usepackage{soul,color}
\usepackage{accents}
\usepackage{enumerate}
\usepackage[utf8x]{inputenc}
\usepackage[T1]{fontenc}
\usepackage[english,francais]{babel}
\usepackage{amsfonts, amscd, amsmath, amssymb, amsthm, mathrsfs, mathtools}
\usepackage{braket}
\input xy
\xyoption{all}
\usepackage[utf8x]{inputenc}
\usepackage[T1]{fontenc}
\usepackage[english]{babel}
\usepackage{amsfonts, amscd, amsmath, amssymb, amsthm, mathrsfs, mathtools}
\usepackage{graphicx}
\usepackage{subfig}
\usepackage{braket}
\usepackage[left=2.5cm,right=2.5cm,top=2.5cm,bottom=2.5cm]{geometry}
\usepackage{fancyhdr}
\theoremstyle{definition}

\newtheorem{remark}{Remark}[section]

\newtheorem{notazione}{Notation}[section]

\newtheorem{question}{Question}[section]
\newtheorem{definizione}{Definition}[section]

\theoremstyle{plain}
\newtheorem{teorema}{Theorem}[section]
\newtheorem{proposizione}{Proposition}[section]
\newtheorem{lemma}{Lemma}[section]
\newtheorem{corollario}{Corollary}[section]

\newcommand{\numberset}{\mathbb}

\newcommand{\R}{\numberset{R}}
\newcommand{\N}{\numberset{N}}
\newcommand{\Z}{\numberset{Z}}
\newcommand{\A}{\mathbb{A}}
\newcommand{\T}{\mathbb{T}}

\usepackage{bookmark,hyperref}

\DeclarePairedDelimiter{\abs}{\lvert}{\rvert}
\DeclarePairedDelimiter{\norma}{\lVert}{\rVert}

\makeatletter
\let\oldabs\abs
\def\abs{\@ifstar{\oldabs}{\oldabs*}}
\let\oldnorma\norma
\def\norma{\@ifstar{\oldnorma}{\oldnorma*}}

\title{Torsion of instability zones for conservative twist maps on the annulus}

\author[Florio]{Anna Florio$^1$}
\address{$^1$Sorbonne Université, Université de Paris, CNRS, Institut de 
	Mathématiques de Jussieu-Paris Rive Gauche, F-75005 Paris, France}
\email{anna.florio@imj-prg.fr}

\author[Le Calvez]{Patrice Le Calvez$^2$}
\address{$^2$Sorbonne Université, Université de Paris, CNRS, Institut de 
	Mathématiques de Jussieu-Paris Rive Gauche, F-75005 Paris, France et 
	Institut Universitaire de France}
\email{patrice.le-calvez@imj-prg.fr}
\date{\today}

\begin{document}

\selectlanguage{english}
\maketitle
\begin{abstract}
	\noindent For a twist map $f$ of the annulus preserving the Lebesgue measure, we give sufficient conditions to assure the existence of a set of positive measure of points with non-zero asymptotic torsion. In particular, we deduce that every bounded instability region for $f$ contains a set of positive measure of points with non-zero asymptotic torsion. Moreover, for an exact symplectic twist map $f$, we provide a simple, geometric proof of a result by Cheng and Sun (see \cite{ChengSun}) which characterizes $\mathcal{C}^0$-integrability of $f$ by the absence of conjugate points.
\end{abstract}

\section{Introduction}

\indent Let us denote $\T=\R/\Z$ and $\A=\T\times\R$. Fix the standard Riemannian metric and the standard trivialization on $\A$. This enables us to define the notion of oriented angle between two non-zero vectors in the tangent plane of a point of $\A$.

Let $\R^2\ni(x,y)\mapsto \pi(x,y)=(x+\Z,y)\in\A$ be the covering projection on $\A$. Let $p_1:\A\rightarrow\T$ and $p_2:\A\rightarrow\R$ denote the projections over the first and the second coordinate respectively. With an abuse of notation, we denote as $p_1,p_2:\R^2\rightarrow\R$ also the projections over the first and the second coordinate on the plane.

Denote the area form $\omega=dx\wedge dy$ and let $\text{Leb}$ be the associated Lebesgue measure on $\A$. Let $\alpha=y\,dx$ be the Liouville 1-form on $\A$. Let $T^1\A$ be the bundle of half-lines, that is
$$
T^1\A=\{(z,w)\in T\A :\ w\neq 0\} / \sim
$$
where $(z,w)\sim(z',w')$ if and only if $z=z'$ and $w'=\lambda w$ for some $\lambda\in\R,\lambda>0$. For every $z\in\A$, $T_z^1\A$ is the fiber of $z$ and for a vector $w\in T_z\A$ we will denote, with an abuse of notation, as $w\in T_z^1\A$ the equivalent class to which $w$ belongs. For every $X\subset\A$ denote as $T^1X$ the restriction of the bundle to $X$, that is $T^1X=\bigcup_{z\in X}T_z^1\A$.

For every $z\in\A$ denote as $v$ the class of the vertical vector $v=(0,1)$. For every $z\in\A$ and every $w,w'\in T^1_z\A$ let $\prec(w,w')$ be the oriented angle between $w$ and $w'$ divided by $2\pi$. We consider angles in $\T=\R/\Z$, after rescaling $\R/2\pi\Z$. Define the continuous function
\begin{eqnarray}\label{or angle}
\theta:&T^1\A\longrightarrow \T \nonumber \\
&(z,w)\longmapsto\prec(v,w).
\end{eqnarray}
It induces a closed 1-form $d\theta$ on $T^1\A$, see \cite{godbillon}. For any continuous path $[0,1]\ni t\mapsto \gamma(t)\in T^1\A$ denote
\begin{equation}\label{def path angle}
\int_{\gamma}d\theta=\Theta(1)-\Theta(0),
\end{equation}
where $\Theta:[0,1]\rightarrow\R$ is a lift of $\theta\circ\gamma$ (the definition in \eqref{def path angle} does not depend on the chosen lift $\Theta$). Let $f:\A\rightarrow\A$ be a $\mathcal{C}^1$ diffeomorphism  isotopic to the identity and $I=(f_t)_{t\in[0,1]}$ an isotopy in $\text{Diff }^1(\A)$ joining the identity to $f_1=f$. Denote as $DI=(Df_t)_{t\in[0,1]}$ the continuous isotopy such that for every $t\in[0,1]$
\begin{eqnarray*}
Df_t:&T^1\A\longrightarrow  T^1\A \\
&(z,w)\mapsto Df_t(z,w)= \left(f_t(z),Df_t(z)w\right).
\end{eqnarray*}
Extend the isotopy for $t\in \R_+$ so that $Df_{1+t}=Df_t\circ Df_1$.\\

\indent We are interested in the notion of torsion of a point: roughly speaking it is the asymptotic rotational behavior of vectors in the tangent space of the point under the action of the linearized dynamics. See \cite{ruelle} and \cite{beguin}.

More precisely, let $(z,w)\in T^1\A$. For every $n\in\N^*$ denote $DI^n(z,w)$ the continuous path $[0,1]\ni t\mapsto Df_{nt}(z,w)$. The {\it{torsion at finite time}} $n\in\N^*$ at $(z,w)\in T^1\A$ is
\begin{equation}\label{def torsion finite}
\text{Torsion}_n(f,z,w) :=\dfrac{1}{n}\displaystyle{\int_{DI^n(z,w)}}d\theta=\dfrac{1}{n}\sum_{i=0}^{n-1}\text{Torsion}_1(f,Df^i(z,w)).
\end{equation}
It does not depend on the choice of the isotopy joining the identity to $f$ and can be equivalently given by using the universal covering space of $T^1\A$ and a continuous lift $\tilde{\theta}:\widetilde{T^1\A}\rightarrow\R$ of the function $\theta$ in \eqref{or angle}, see \cite{beguin} and \cite{Flo19}.

The {\it{asymptotic torsion}} at $z\in\A$ is, whenever the limit exists,
\begin{equation}
\text{Torsion}(f,z):=\lim_{n\rightarrow+\infty}\text{Torsion}_n(f,z,w).
\end{equation}
\indent The definition of asymptotic torsion at $z\in\A$ coincides with Ruelle's definition of rotation number, see \cite{ruelle}. It does not depend on $w\in T^1_z\A$. Indeed, let $z\in\A$, $n\in\N^*$ and $w,w'\in T^1_z\A$. Consider the quantity
\begin{equation}\label{eq prop vec}
\displaystyle{\int_{DI^n(z,w)}}d\theta-\displaystyle{\int_{DI^n(z,w')}}d\theta.
\end{equation}
If $\prec(w,w')\in(0,\frac{1}{2})+\Z$, then $\prec (Df^n(z)w,Df^n(z)w')\in(0,\frac{1}{2})+\Z$ because $Df(z)$ preserves the orientation and so the quantity \eqref{eq prop vec} cannot be equal to $\frac{1}{2}+k,k\in\Z$. This is also true if $\prec(w,w')\in(-\frac{1}{2},0)+\Z$. Moreover \eqref{eq prop vec} vanishes if $\prec(w,w')\in\{0,\frac{1}{2}\}+\Z$, meaning $w=w'$ or $w=-w'$. By connectedness, for any $n\in\N^*$ and for every $w,w'\in T^1_z\A$ it holds
\begin{equation}\label{eq check control}
-\dfrac{1}{2}<\displaystyle{\int_{DI^n(z,w)}}d\theta-\displaystyle{\int_{DI^n(z,w')}}d\theta<\dfrac{1}{2}.
\end{equation}
Observe that the function $T^1\A\times\Z\ni (z,w,n)\mapsto \int_{DI^n(z,w)}d\theta\in\R$ is a cocycle over $Df$.

Assume that $f$ admits an invariant Borel probability measure $\mu$ on $\A$ with compact support. There exists a $Df$-invariant probability measure $\nu$ on $T^1\A$ whose projection on $\A$ is $\mu$. Indeed, if $\nu_0$ is a probability measure of $T^1\A$ that projects onto $\mu$, any limit point $\nu$ of the sequence $\left( \frac{1}{n}\sum_{i=0}^{n-1}Df_*^i(\nu_0) \right)_{n\geq 1}$ is $Df$-invariant and projects onto $\mu$.

Since the asymptotic torsion is the limit of the Birkhoff's sum in \eqref{def torsion finite}, by Birkhoff's Ergodic Theorem, the limit exists at $\nu$-almost every $(z,w)\in T^1\A$. Since the asymptotic torsion does not depend on the chosen $w\in T_z^1\A$ and since $\nu$ projects onto $\mu$, we deduce that the asymptotic torsion exists at $\mu$-almost every $z\in\A$, see \cite{ruelle}.
\begin{definizione}
	A {\it{positive (respectively negative) twist map}} is a $\mathcal{C}^1$ diffeomorphism $f:\A\rightarrow\A$ isotopic to the identity such that for any lift $F:\R^2\rightarrow\R^2$ of $f$ and for any $x\in\R$ the function
	$$
	\R\ni y\mapsto p_1\circ F(x,y)\in\R
	$$
	is an increasing (respectively decreasing) diffeomorphism of $\R$.
\end{definizione}
\indent It can be easily proved that $f$ is a positive twist map if and only if $f^{-1}$ is a negative twist map. Concerning the torsion for positive (respectively negative) twist map, for every $z\in\A$ and for every $n\in\N^*$ it holds $\text{Torsion}_n(f,z,v)\in (-\frac{1}{2},0)$ (respectively $(0,\frac{1}{2})$), see \cite{Flo19}. It is an outcome of the fact that, for a positive twist map, $\text{Torsion}_1(f,z,v)\in(-\frac{1}{2},0)$ and that for every isotopy $(f_t)_{t\in[0,1]}$ from $\text{Id}$ to $f$ and for every non colinear vectors $w,w'\in T_z\A$, the vectors $Df_t(z)w,Df_t(z)w'\in T_{f_t(z)}\A$ are non colinear and if $\prec(w,w')\in\left(0,\frac{1}{2}\right)+\Z$, then $\prec(Df_t(z)w,Df_t(z)w')\in\left(0,\frac{1}{2}\right)+\Z$.
\begin{definizione}\label{conj pts}
	Let $f:\A\rightarrow\A$ be a positive twist map. A point $z\in\A$ has {\it{conjugate points}} if there exist $n,k\in\N^*$ such that
	$\int_{DI^n(z,v)}d\theta =-\frac{k}{2}$, it has {\it{over-conjugate points}} if there exists $n\in\N^*$ such that
	$\int_{DI^n(z,v)}d\theta< -\frac{1}{2}$.
\end{definizione}
Observe that if a point $z\in\A$ has negative asymptotic torsion, then $z$ has over-conjugate points. Actually, the existence of points with conjugate points and with over-conjugate points is equivalent, see Lemma \ref{equiv overconj et conj}.

\indent Let us state now the main results of the article. The first result, that will be fundamental, asserts that for a twist map that preserves the Lebesgue measure, the existence of points having over-conjugate points in a $f$-periodic open domain $U$ of finite measure is equivalent to the fact that the set of points in $U$ of non-zero torsion has positive Lebesgue measure. More precisely:
\begin{proposizione}\label{prop intro 1}
	Let $f:\A\rightarrow\A$ be a positive twist map that preserves the Lebesgue measure. Let $U\subset \A$ be a $f$-periodic open set of finite Lebesgue measure. Then the following conditions are equivalent.\vspace{5pt}
	
	\begin{itemize}
		\item[$(i)$] The asymptotic torsion is defined and equal to zero at every point $z\in U$.\vspace{5pt}
		
		\item[$(ii)$] The asymptotic torsion is defined and equal to zero at $\text{Leb}$-almost every point $z\in U$.\vspace{5pt}
		
		\item[$(iii)$] No point $z\in U$ has over-conjugate points.
	\end{itemize}
\end{proposizione}
\indent Recall that an essential curve $\gamma:\T\rightarrow\A$ is a topological embedding that is not homotopic to a point, and that an open subset $U$ of $\A$ is essential if it contains an essential curve.
\noindent We show that any $f$-periodic open subset $U$ of finite measure which is not essential must contain points with over-conjugate points. Actually, we prove the following
\begin{teorema}\label{teo intro 3}
	Let $f$ be a positive twist map that preserves the Lebesgue measure. Let $U$ be a $f$-periodic non essential open set of finite measure. Then the set of points of $U$ of negative asymptotic torsion is dense and has positive measure.
\end{teorema}
\begin{question}
	Let $U$ be a $f$-periodic non essential open set of finite Lebesgue measure. Is the measure of the set of points of non zero torsion in $U$ equal to the measure of $U$?
\end{question}
\begin{definizione}
	Let $U\subset \A$ be an $f$-invariant essential subannulus. The dynamics $f_{\vert U}$ is {\it{$\mathcal{C}^0$-integrable}} if $f_{\vert U}$ admits a partition into essential $f$-invariant curves.
\end{definizione}
Observe that the definition of $\mathcal{C}^0$-integrability is more restrictive than the usual notion of integrability. In higher dimensions, for Tonelli Hamiltonian flows on the cotangent bundle $T^*\mathbb{T}^n$ of the $n$-dimensional torus, the analogous definition appears in \cite{AABZ}.\\
\indent Recall that a $\mathcal{C}^1$ diffeomorphism $f:\A\rightarrow\A$ that preserves the Lebesgue measure is {\it{exact symplectic}} if $f^*\alpha-\alpha$ is exact, or equivalently if for every essential curve $\gamma:\T\rightarrow\A$, the algebraic area contained between $\gamma(\T)$ and $f\circ\gamma(\T)$ is zero (zero flux condition).

We provide a simple proof of a result by Cheng and Sun, see Theorem 1 in \cite{ChengSun}.
\begin{teorema}\label{teo intro 2}
	An exact symplectic twist map $f:\A\rightarrow\A$ is $\mathcal{C}^0$-integrable if and only if $f$ does not have points with conjugate points.
\end{teorema}
\indent In the continuous case, the same result holds for Tonelli Hamiltonian flows on the cotangent bundle of the $n$-dimensional torus, see Theorem 1 in \cite{AABZ}. Cheng and Sun's proof uses Aubry-Mather theory as presented by Bangert in \cite{bangert}, that is considering minimizing configurations. Their ideas can be easily adapted to any bounded invariant subannulus, see \cite{FloPHD}. Our proof is a geometrical one and it works both for exact symplectic twist maps on the unbounded annulus and for conservative twist maps on bounded essential subannuli, see Remark \ref{adapt bounded annulus}. Thus, thanks to Proposition \ref{prop intro 1}, we deduce the following
\begin{corollario}
	Let $f:\A\rightarrow\A$ be a twist map that preserves the Lebesgue measure. Let $U$ be an $f$-invariant bounded essential subannulus. Then, $f_{\vert U}$ is $\mathcal{C}^0$-integrable if and only if $\text{\emph{Leb}}$-almost every point of $U$ has zero torsion.
\end{corollario}
\indent Denote as $\mathscr{I}(f)$ the union of all $f$-invariant essential curves in $\A$. A bounded connected component of $\A\setminus\mathscr{I}(f)$ is either an $f$-invariant essential subannulus or a $f$-periodic disk, see \cite{mc_greenbundles}. Thus, from Proposition \ref{prop intro 1} and Theorems \ref{teo intro 2} and \ref{teo intro 3}, we deduce the following outcome.
\begin{corollario}
	Let $f:\A\rightarrow\A$ be a positive twist map that preserves the Lebesgue measure. Then every bounded connected component $U$ of $\A\setminus\mathscr{I}(f)$ contains a set of positive measure of points with non-zero torsion. In particular, it holds
	$$
	\displaystyle{\int_{\bigcup_{i\in\Z}f^i(U)}}\text{\emph{Torsion}}(f,z)\, d\text{\emph{Leb}}(z)<0.
	$$
\end{corollario}

\section{Torsion and over-conjugate points: proof of Proposition \ref{prop intro 1}}

\indent The implications $(iii)\Rightarrow (i)$ and $(i)\Rightarrow (ii)$ of Proposition \ref{prop intro 1} are obvious. The implication $(ii)\Rightarrow (iii)$ is an immediate consequence of the following result.
\begin{proposizione}\label{prop 2.1}
	Let $f:\A\rightarrow\A$ be a positive twist map that preserves the Lebesgue measure. Let $z_0\in\A$ be such that\vspace{5pt}
	
	\begin{itemize}
		\item[$(i)$] $z_0$ has over-conjugate points;\vspace{5pt}
		
		\item[$(ii)$] $z_0$ has a neighborhood $U$ such that $\text{\emph{Leb}}(\bigcup_{n\in\Z}f^n(U))<+\infty$.\vspace{5pt}
		
	\end{itemize}
Then there exists an open neighborhood $W\subset U$ of $z_0$ such that for $\text{\emph{Leb}}$-almost every $z\in W$ it holds that $\text{\emph{Torsion}}(f,z)<0$.
\end{proposizione}

\begin{proof}
	By the definition of points having over-conjugate points, there exists $m\in\N^*$ such that $\int_{DI^m(z_0,v)}d\theta<-\frac{1}{2}$. By the continuity of finite-time torsion, there exists a neighborhood $W\subset U$ of $z_0$ such that for every $z\in W$ it holds that $\int_{DI^m(z,v)}d\theta<-\frac{1}{2}$.
	
	For every $z\in W$ and for every $n>m$, we have that  $\int_{DI^n(z,v)}d\theta<-\frac{1}{2}$. This fact can be shown by induction. Indeed, for $n=m+1$, since $Df(z)$ preserves the orientation and since $\int_{DI(f^m(z),v)}d\theta<0$, we deduce that, if $\int_{DI^m(z,v)}d\theta <-\frac 1 2$, then
	$$
	\int_{DI^{m+1}(z,v)}d\theta \leq -\dfrac 1 2 +\int_{DI(f^m(z),v)}d\theta<-\dfrac 1 2.
	$$
	We argue similarly for $n>m+1$. Actually, it holds that
	$$
	\displaystyle{\int_{DI^n(z,v)}}d\theta=\displaystyle{\int_{DI^m(z,v)}}d\theta+\displaystyle{\int_{DI^{n-m}(Df^m(z,v))}}d\theta\leq -\dfrac{k}{2}+\displaystyle{\int_{DI^{n-m}(f^m(z),v)}}d\theta<-\dfrac{1}{2},
	$$
	where $k\in\N, k>0$ is the integer part of $-2\int_{DI^m(z,v)}d\theta$.
	
	Thus, by \eqref{eq check control}, for every $(z,w)\in T^1W$ and for every $n\geq m$ it holds that
	\begin{equation}\label{strict neg}
	\displaystyle{\int_{DI^n(z,w)}}d\theta <0.
	\end{equation}
	There exists a finite measure $\widetilde{\text{Leb}}$ defined on $T^1(\cup_{n\in\Z}f^n(U))$ that projects onto the restriction of $\text{Leb}$ to $\bigcup_{n\in\Z}f^n(U)$. Indeed, the argument given in the introduction can be extended to the case of an open set of finite measure: we can either consider the Alexandroff compactification of $T^1(\bigcup_{n\in\Z}f^n(U))$, or we can use Prohorov's Theorem (see Proposition 2.4.1 in \cite{ViaOli}).
	
	Since $f$ is a positive twist map, one has
	$$
	-1<\text{Torsion}_1(f,z,w)<\dfrac{1}{2}
	$$
	for every $(z,w)\in T^1\A$. By Birkhoff's Ergodic Theorem, $\lim_{n\rightarrow+\infty}\text{Torsion}_n(f,z,w)$ is defined $\widetilde{\text{Leb}}$-almost everywhere on $T^1(\bigcup_{n\in\Z}f^n(U))$ and so $\text{Torsion}(f,z)$ is defined $\text{Leb}$-almost everywhere on $\bigcup_{n\in\Z}f^n(U)$.
	
	We want to prove that $\text{Torsion}(f,z)<0$ for $\text{Leb}$-almost every point in $W$. Set $g=f^m$ and $J=I^m$. We will prove that $\text{Torsion}(g,z)$ is well-defined and negative for $\text{Leb}$-almost every $z\in W$, which will imply the proposition because $\text{Torsion}(f,z)=\text{Torsion}(g,z)/m$.
	
	 Let $\psi:z\mapsto g^{\tau(z)}(z)$ be the first return map of $g$ on $W$, where $\tau:W\rightarrow\N^*$ is the first return time map. These maps are defined $\text{Leb}$-almost everywhere on $W$. The map $D\psi:(z,w)\mapsto (g^{\tau(z)}(z),Dg^{\tau(z)}(z)w)$ is defined $\widetilde{\text{Leb}}$-almost everywhere on $T^1W$. We know that
	\begin{equation}\label{bbound}
	-m\tau(z)<\displaystyle{\int_{DJ^{\tau(z)}(z,w)}}d\theta=\sum_{i=0}^{\tau(z)-1}\displaystyle{\int_{DJ(Dg^i(z,w))}}d\theta<0,
	\end{equation}
	the right inequality occuring because $\int_{DI^n(z,w)}d\theta<0$ if $n\geq m$ and $DJ^{\tau(z)}(z,w)=DI^{m\tau(z)}(z,w)$ (see \eqref{strict neg}) and the left inequality occuring because $\text{Torsion}_1(f,z,w)>-1$ for every $(z,w)\in T^1\A$. By Kac's Lemma we know that $\tau$ is $\text{Leb}$-integrable and so that $\phi:(z,w)\mapsto \int_{DJ^{\tau(z)}(z,w)}d\theta$ is integrable and negative. By Birkhoff's Ergodic Theorem, we deduce that for $\widetilde{\text{Leb}}$-almost every $(z,w)\in T^1W$, one has
	\begin{equation}\label{eq bar phi}
	\lim_{n\rightarrow+\infty}\dfrac{1}{n}\sum_{i=0}^{n-1}\phi(D\psi^i(z,w))=\bar{\phi}(z,w)\in(-\infty,0)
	\end{equation}
	and
	\begin{equation}\label{eq bar tau}
	\lim_{n\rightarrow+\infty}\dfrac{1}{n}\sum_{i=0}^{n-1}\tau(\psi^i(z))=\bar{\tau}(z)\in[1,\infty).
	\end{equation}
	Recall that the fact that $\bar{\phi}(z,w)<0$ in \eqref{eq bar phi} is a classical application of Birkhoff's Ergodic Theorem. Consequently, from \eqref{bbound},  \eqref{eq bar phi} and \eqref{eq bar tau}, for $\widetilde{\text{Leb}}$-almost every $(z,w)\in T^1W$, it holds that $\bar{\phi}(z,w)/\bar{\tau}(z)\in(-\infty,0)$. Note now that
	$$
	\sum_{i=0}^{n-1}\phi(D\psi^i(z,w))/\sum_{i=0}^{n-1}\tau(\psi^i(z))=\text{Torsion}_{N_n}(g,z,w)
	$$
	where $N_n=\sum_{i=0}^{n-1}\tau(\psi^i(z))$. So $\text{Torsion}(g,z)=\lim_{n\rightarrow+\infty}\text{Torsion}_{N_n}(g,z,w)<0$ for $\text{Leb}$-almost every point in $W$.\\
\end{proof}

\section{Density of over-conjugate points in non-essential periodic bounded open sets: proof of Theorem \ref{teo intro 3}}\label{sect 3}

\indent In the introduction, we give the definition of \emph{conjugate} and \emph{over-conjugate} points. Let us show that the existence of the first ones is equivalent to the existence of the second ones.
\begin{lemma}\label{equiv overconj et conj}
	Let $f$ be a positive twist map. There exists a point with conjugate points if and only if there exists a point with over-conjugate points.
\end{lemma}
\begin{proof}
 On one hand, the first implication is an outcome of the twist property. Indeed, if there exists $z\in\A$ such that $\int_{DI^n(z,v)}d\theta=-\frac{k}{2}$ for some $n,k\in\N^*$, then $Df^n(z)v=\pm v$ and so $\int_{DI(f^n(z),Df^n(z)v)}d\theta\in(-\frac{1}{2},0)$. Consequently, one has
$$
\int_{DI^{n+1}(z,v)}d\theta=\int_{DI^n(z,v)}d\theta+\int_{DI(f^n(z),Df^n(z)v)}d\theta<-\dfrac{k}{2}\leq -\dfrac{1}{2}.
$$

On the other hand, it can be proved that, for every $n\in\N^*$, there exists $z\in\A$ such that $\int_{DI^n(z,v)}d\theta\in(-\frac{1}{2},0)$, see \cite{Flo20}. More precisely, selecting a lift $F$ of $f$, the linking number at any finite time $n\in\N^*$ of the couple of points $(x,r_0),(x+1,r_0)$, for some $x\in\R$, $r_0\in\R$, is zero. Thanks to the relation between torsion and linking number shown in \cite{Flo19}, we deduce that, for any $n\in\N^*$ there exists $z=z(n)\in\T\times\{r_0\}$ such that $\int_{DI^n(z,u)}d\theta=0$, where $u$ is the class of the horizontal vector $u=(1,0)$. By \eqref{eq check control} and since $f$ is a positive twist map, we so have that $\int_{DI^n(z,v)}d\theta\in(-\frac{1}{2},0)$. By this fact and by the Intermediate Value Theorem, we conclude that the existence of a point with over-conjugate points implies the existence of a point with conjugate points.
\end{proof}

The equivalence between the existence of points with conjugate points and with over-conjugate points still holds for composition of positive twist maps and even to the restriction of the dynamics to a $f$-invariant essential subannulus.

\indent This section is devoted to the proof of Theorem \ref{teo intro 3}.

\begin{lemma}\label{lemma first sec 3}
	Let $f$ be a positive twist map, $F$ a lift of $f$ to $\R^2$ and $K\subset\A$ a compact set that does not contain any point with over-conjugate points. Then there exists $\beta\in(0,\frac{1}{2})$ so that for every $z,z'\in\pi^{-1}(K)$ satisfying\vspace{3pt}
	
	\begin{itemize}
	\item[$(i)$] the angle $\prec(v,z'-z)$ belongs to $[0,\beta]+\Z$,\vspace{3pt}
		
	\item[$(ii)$] the segment joining $z$ to $z'$ is contained in $\pi^{-1}(K)$,\vspace{3pt}
		
	\end{itemize}
we have that for every $n\geq 1$
$$
p_1\circ F^n(z')-p_1\circ F^n(z)>0.
$$
\end{lemma}

\begin{proof}
The map $f$ being a positive twist map, we have $\prec(v,Df(z)v)\in\left(-\frac{1}{2},0\right)+\Z$, for every $z\in\A$. By compactness of $K$, there exists $\beta\in\left(0,\frac{1}{2}\right)$ such that for every $(z,w)\in T^1K$ such that $\prec(v,w)\in[0,\beta]+\Z$ it holds that $\prec(v,Df(z)w)\in\left(-\frac{1}{2},0\right)+\Z$. Moreover we know that $\prec(Df(z)v,Df(z)w)\in\left(0,\frac{1}{2}\right)+\Z$. We deduce that for every $n\geq 0$ we have $\prec(Df^n(z)v,Df^{n+1}(z)w)\in\left(-\frac{1}{2},0\right)+\Z$ and $\prec(Df^{n+1}(z)v,Df^{n+1}(z)w)\in\left(0,\frac{1}{2}\right)+\Z$. Since there are no points with over-conjugate points in $K$, we have $\prec(v,Df^{n+1}(z)v)\in\left(-\frac{1}{2},0\right)+\Z$ for every $n\geq 0$. Consequently we have $\prec(v,Df^{n+1}(z)w)\in\left(-\frac{1}{2},0\right)+\Z$ for every $n\geq 0$.

Let $z,z'\in\pi^{-1}(K)$ be such that $\prec(v,z'-z)\in[0,\beta]+\Z$ and the segment $\gamma:[0,1]\ni t\mapsto (1-t)z+tz'$ is contained in $\pi^{-1}(K)$. Consequently, for every $t$ and for every $n\geq 1$ $\prec(v,(F^n\circ\gamma)'(t))\in(-\frac{1}{2},0)+\Z$. So for every $n\geq 1$ it holds that $$p_1\circ F^n(\gamma(1))-p_1\circ F^n(\gamma(0))=p_1\circ F^n(z')-p_1\circ F^n(z)>0.$$
\end{proof}

\noindent \textit{Proof of Theorem \ref{teo intro 3}.} Let $U$ be a $f$-periodic non essential open set of period $M\in\N^*$. Let $F:\R^2\rightarrow\R^2$ be a lift of $f$. Let $\mathscr{U}=\pi^{-1}(U)$. By Proposition \ref{prop 2.1}, to get Theorem \ref{teo intro 3}, it is sufficient to prove that the set of points with over-conjugate points is dense in $U$.

Argue by contradiction and assume there is a closed ball $B\subset \mathscr{U}$ which does not have any point with over-conjugate points.

 Let $\beta\in(0,\frac{1}{2})$ be given by Lemma \ref{lemma first sec 3} applied at the closed ball $B$. We can find two balls $B_1,B_2\subset B$ such that for every $z\in B_1,z'\in B_2$ the oriented angle $\prec(v,z'-z)$ belongs to $[0,\beta]+\Z$ and so that the segment joining $z$ to $z'$ is contained in $B$. On one hand, by Lemma \ref{lemma first sec 3}, for every $n\geq 1$ we have that
$$
p_1\circ F^n(z')-p_1\circ F^n(z)>0.
$$
On the other hand, on the annulus, we apply Poincaré's Recurrence Theorem to the map $f\times f$ and the set $\pi(B_1)\times\pi(B_2)$. We find $n\in\N^*$, $z\in B_1,z'\in B_2$ such that $f^{Mn}(\pi(z))\in\pi(B_1)$ and $f^{Mn}(\pi(z'))\in\pi(B_2)$. Thus, there exists $k_1,k_2\in\Z$ such that $F^{Mn}(z)+(k_1,0)\in B_1,F^{Mn}(z')+(k_2,0)\in B_2$. Since $B\subset \mathscr{U}$ and since $U$ is non essential, we deduce that $k_1=k_2$ and, by the choice of $B_1,B_2$ it holds that $p_1\circ F^{Mn}(z')<p_1\circ F^{Mn}(z)$, obtaining the required contradiction.\\
\hfill\qed
\begin{figure}[h]
	\begin{center}
		\includegraphics[scale=.45]{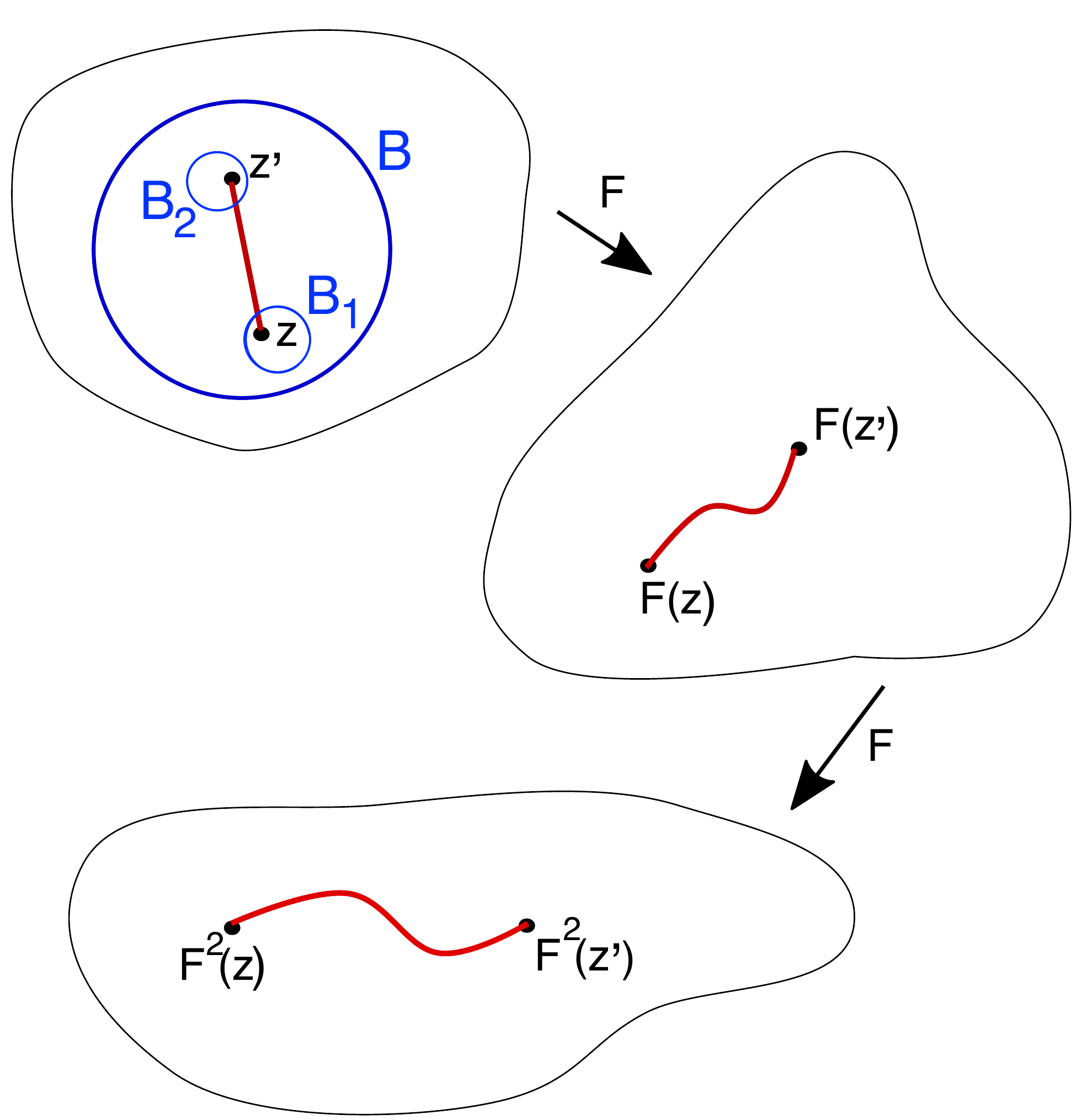}
		\caption{Each vector $DF^n(\gamma(t))\gamma'(t)$ lies in the right half-plane.}
		\label{Figure disk}
	\end{center}
\end{figure}

\begin{remark}
	The idea of the proof of Theorem \ref{teo intro 3} can be pushed further. Actually, for a bounded invariant topological open disk the absence of points with over-conjugate points allows to build a partition into partial graphs of Lipschitz functions which are periodic. See \cite{FloPHD} for such a construction.
	
	 We remark also that an alternative proof of Theorem \ref{teo intro 3} can be given by using the notion of linking number (see \cite{beguin} and \cite{Flo19}) and thanks to the relation between torsion and linking number proved in \cite{Flo19}. Indeed, since $U$ is non essential, denoting as $n\in\N$ the period of $U$, we can select a lift $F$ of $f$ such that $F^n$ fixes the lifts of $U$. In the lifted framework, we can consider the linking number at finite-time of a couple of points, calculated with respect to an isotopy obtained by lifting an isotopy on $\A$.
	 
	 Apply then Poincare's Recurrence Theorem at any small ball contained in $U$ with respect to the dynamics of $f_{\vert U}\times f_{\vert U}$. Thanks to the twist property, we find a couple of points whose linking number at finite-time is almost $-1$. Consequently, by the result in \cite{Flo19}, there exists a point belonging to the starting small ball whose torsion at finite-time is almost $-1$. That is, any arbitrary small ball contains points with over-conjugate points.\\
\end{remark}

\section{A geometric criterion for $\mathcal{C}^0$-integrability: proof of Theorem \ref{teo intro 2}}
\indent The aim of this section is giving a simple, geometrical proof of the result due to Cheng and Sun, see \cite{ChengSun} and here Theorem \ref{teo intro 2}. We prove here the non trivial implication, i.e. if there are no points having conjugate points, then $f_{\vert U}$ is $\mathcal{C}^0$-integrable.
\begin{remark}
The definition given in \cite{ChengSun} of points with conjugate points is equivalent to our definition of conjugate points (here Definition \ref{conj pts}).

Let $F$ be a lift of the exact symplectic twist map $f$ and let $h:\R^2\rightarrow\R$ be a generating function for $F$. A sequence $(x_n,y_n)_{n\in\N}\subset (\R^2)^{\N}$ is a trajectory for $F$ if and only if the sequence $(x_n)_{n\in\N}$ satisfies 
$$
\partial_1h(x_n,x_{n+1})+\partial_2 h(x_{n-1},x_n)=0\quad\forall n\in\N,
$$
see \cite{bangert}, \cite{aubry} or \cite{mackay}. Consequently, when looking at the linearized dynamics, a sequence $(x_n,y_n;\xi_n,\eta_n)_{n\in\N}\subset (T\R^2)^{\N}$ is a trajectory of the linearized map $DF$ if and only if the sequence $(\xi_n)_{n\in\N}$ is a Jacobi vector field over $(x_n)_{n\in\N}$, i.e.
$$
\partial_1\partial_2 h(x_{n-1},x_n)\xi_{n-1}+[\partial^2_1 h(x_n,x_{n+1})+\partial^2_2h(x_{n.1},x_n)]\xi_n+\partial_1\partial_2 h(x_n,x_{n+1})\xi_{n+1}=0.
$$
In \cite{ChengSun}, a point $(x_0,y_0)$ has conjugate points if there exists a non-zero Jacobi vector field $(\xi_n)_{n\in\N}$ over $(x_n)_{n\in\N}=(p_1\circ F^n(x_0,y_0))_{n\in\N}$ such that $\xi_0=\xi_N=0$. Equivalently, $DF^N(x_0,y_0)v=\pm v$. Since $f$ is a positive twist map, it means that $\int_{DI^N((x_0,y_0), v)}d\theta=-\frac k 2$ for some $k\in\N^*$, i.e. $(x_0,y_0)$ has conjugate points according to Definition \ref{conj pts}.
\end{remark}

We recall the definition of rotation number for a homeomorphism of the annulus.
\begin{definizione}
	Let $F:\R^2\rightarrow\R^2$ be the lift of a homeomorphism $f$ of $\A$ isotopic to the identity. Let $z\in\R^2$. The rotation number for $F$ of $z$ is, whenever it exists, the limit
	$$
	\rho_F(z)=\lim_{n\rightarrow+\infty}\dfrac{p_1\circ F^n(z)-p_1(z)}{n}.
	$$
\end{definizione}
\begin{remark}
	Let $\psi:\T\rightarrow\R$ be a continuous map whose graph is $f$-invariant and let $\Psi=\psi\circ \pi'$, where $\R\ni x\mapsto \pi'(x)=x+\Z\in\T$ is the covering projection on $\T$. Then, we can define the rotation number $\rho_F(\Psi)$ for $F$ of $\Psi$ as the rotation number $\rho_F(x,\Psi(x))$ for any $x\in\R$. Indeed, for every $x\in\R$ the rotation number $\rho_F(x,\Psi(x))$ exists and moreover it is independent from $x\in\R$.
\end{remark}
\begin{lemma}\label{lemma graph Cg}
	Let $F:\R^2\rightarrow\R$ be the lift of a positive twist map $f$. The set \begin{equation}\label{def ctilde}\tilde{C}_{F}:=\{ z\in \R^2:\ p_1\circ F(z)=p_1(z) \}\end{equation} is the lift of the graph of a function $\psi_1:\T\rightarrow\R$ of class $\mathcal{C}^1$.
\end{lemma}
\begin{proof}
	For every given $x\in\R$, by the twist property, the function $\R\ni y\mapsto p_1\circ F(x,y)\in\R$ is an increasing diffeomorphism and so there exists a unique $y\in\R$ such that $p_1\circ F(x,y)=x$. Let us denote such $y\in\R$ as $\Psi_1(x)$. The set $\tilde{C}_F$ is the graph of the function $\Psi_1:\R\rightarrow\R$.
	
	 To prove that $\Psi_1$ is of class $\mathcal{C}^1$, it is sufficient to apply the Implicit Function Theorem because $\frac{\partial}{\partial y} (p_1\circ F(x,y)-x)>0$. As $\Psi_1$ is clearly periodic of period $1$, it lifts a function $\psi_1:\T\rightarrow\R$ of class $\mathcal{C}^1$.\\
\end{proof}
\begin{remark}
	Lemma \ref{lemma graph Cg} holds true for negative twist maps too. In particular, we denote the set $\tilde{C}_{F^{-1}}=\{z\in\R^2 :\ p_1\circ F^{-1}(z)=p_1(z)\}$ as the graph of the function $\Psi_{-1}:\R\rightarrow\R$ of class $\mathcal{C}^1$ that lifts $\psi_{-1}:\T\rightarrow\R$. Note that $\tilde{C}_{F^{-1}}$ is nothing but the image of $\tilde{C}_F$ by $F$.
	
	 Remark that if a point $z\in\R^2$ lies below (respectively above) $\tilde{C}_F$, then $F(z)$ lies on the left (respectively on the right) of the vertical passing through $z$. Similarly, if $z$ lies below (respectively above) $\tilde{C}_{F^{-1}}$, then $F^{-1}(z)$ lies on the right (respectively on the left) of the vertical passing through $z$.
\end{remark}
\begin{lemma}
	Let $f:\A\rightarrow\A$ be a positive twist map with no point with conjugate points. Then for every $n>0$ the function $f^n$ (respectively $f^{-n}$) is a positive (respectively negative) twist map.
\end{lemma}
\begin{proof}
Let us prove this fact by induction. Suppose that $f^n$ is a positive twist map and fix a lift $F$ of $f$. Consider $x,x'\in\R$. The map $y\mapsto p_1\circ F^n(x,y)$ is an increasing diffeomorphism of the real line, while the map $y\mapsto p_1\circ F^{-1}(x',y)$ is a decreasing one. Moreover, $\lim_{y\rightarrow\pm \infty} p_2\circ F^n(x,y)=\lim_{y\rightarrow\pm\infty}p_2\circ F^{-1}(x',y)=\pm\infty$. This implies that there exist $y,y'\in\R$ such that $F^n(x,y)=F^{-1}(x',y')$. Equivalently, it holds that $F^{n+1}(x,y)=(x',y')$. Thus, the map $y\mapsto p_1\circ F^{n+1}(x,y)$ is surjective. Since there are no points with conjugate points, by Lemma \ref{equiv overconj et conj}, there are no points with over-conjugate points. In particular for any $y\in\R$ it holds that $\prec(v,DF^{n+1}(x,y)v)\in(-\frac{1}{2},0)+\Z$. Thus $\frac{d}{dy}\,p_1\circ F^{n+1}(x,y)>0$ and consequently the map $y\mapsto p_1\circ F^{n+1}(x,y)$ is an increasing diffeomorphism onto its image. So $f^{n+1}$ is a positive twist map.\\
\end{proof}

\begin{lemma}\label{lemma sequences}
	Let $f:\A\rightarrow\A$ be a positive twist map with no point with conjugate points and $F:\R^2\rightarrow\R^2$ be a lift of $f$. Let $x\in\R$.\vspace{3pt}
	\begin{itemize}
	\item[$(i)$] If $\Psi_1(x)<\Psi_{-1}(x)$, then for every $y\in(\Psi_1(x),\Psi_{-1}(x)]$ the sequence $(p_1\circ F^n(x,y))_{n\geq 0}$ is strictly increasing and for every $y\in[\Psi_1(x),\Psi_{-1}(x))$ the sequence $(p_1\circ F^n(x,y))_{n\leq 0}$ is strictly decreasing. Moreover, $(p_1\circ F^n(x,\Psi_1(x)))_{n\geq 1}$ is strictly increasing and $(p_1\circ F^n(x,\Psi_{-1}(x)))_{n\leq-1}$ strictly decreasing.\vspace{3pt}
	
	\item[$(ii)$] If $\Psi_{-1}(x)<\Psi_1(x)$, then for every $y\in[\Psi_{-1}(x),\Psi_1(x))$ the sequence $(p_1\circ F^n(x,y))_{n\geq 0}$ is strictly decreasing and for every $y\in(\Psi_{-1}(x),\Psi_1(x)]$ the sequence $(p_1\circ F^n(x,y))_{n\leq 0}$ is strictly increasing. Moreover, $(p_1\circ F^n(x,\Psi_1(x)))_{n\geq 1}$ is strictly decreasing and $(p_1\circ F^n(x,\Psi_{-1}(x)))_{n\leq -1}$ strictly increasing.
	\end{itemize}
\end{lemma}
\begin{proof}
	Let us consider $x\in\R$ such that $\Psi_1(x)<\Psi_{-1}(x)$, the case $(ii)$ can be treated similarly. Recall that $f^n$ is a positive twist map for every $n\in\N^*$. So, for every $n\geq 1$, the image by $F^n$ of the oriented vertical segment $\{x\}\times[\Psi_1(x),\Psi_{-1}(x)]$ is a path that projects injectively on the horizontal axis, with the first coordinate increasing, and that joins $F^n(x,\Psi_1(x))$ to $F^n(x,\Psi_{-1}(x))=F^{n+1}(x,\Psi_1(x))$. We deduce that for every $y\in(\Psi_1(x),\Psi_{-1}(x)]$, the sequence $(p_1\circ F^n(x,y))_{n\geq 0}$ is strictly increasing and that $(p_1\circ F^n(x,\Psi_1(x)))_{n\geq 1}$ is strictly increasing. See Figure \ref{dessinFLC}. Since $f^n$ is a negative twist map if $n<0$, we prove similarly that $(p_1\circ F^n(x,y))_{n\leq 0}$ is strictly decreasing if $y\in[\Psi_1(x),\Psi_{-1}(x))$, as is $(p_1\circ F^n(x,\Psi_{-1}(x)))_{n\leq -1}$.\\
	\end{proof}
	\begin{figure}[h]
		\centering
		\includegraphics[scale=0.5]{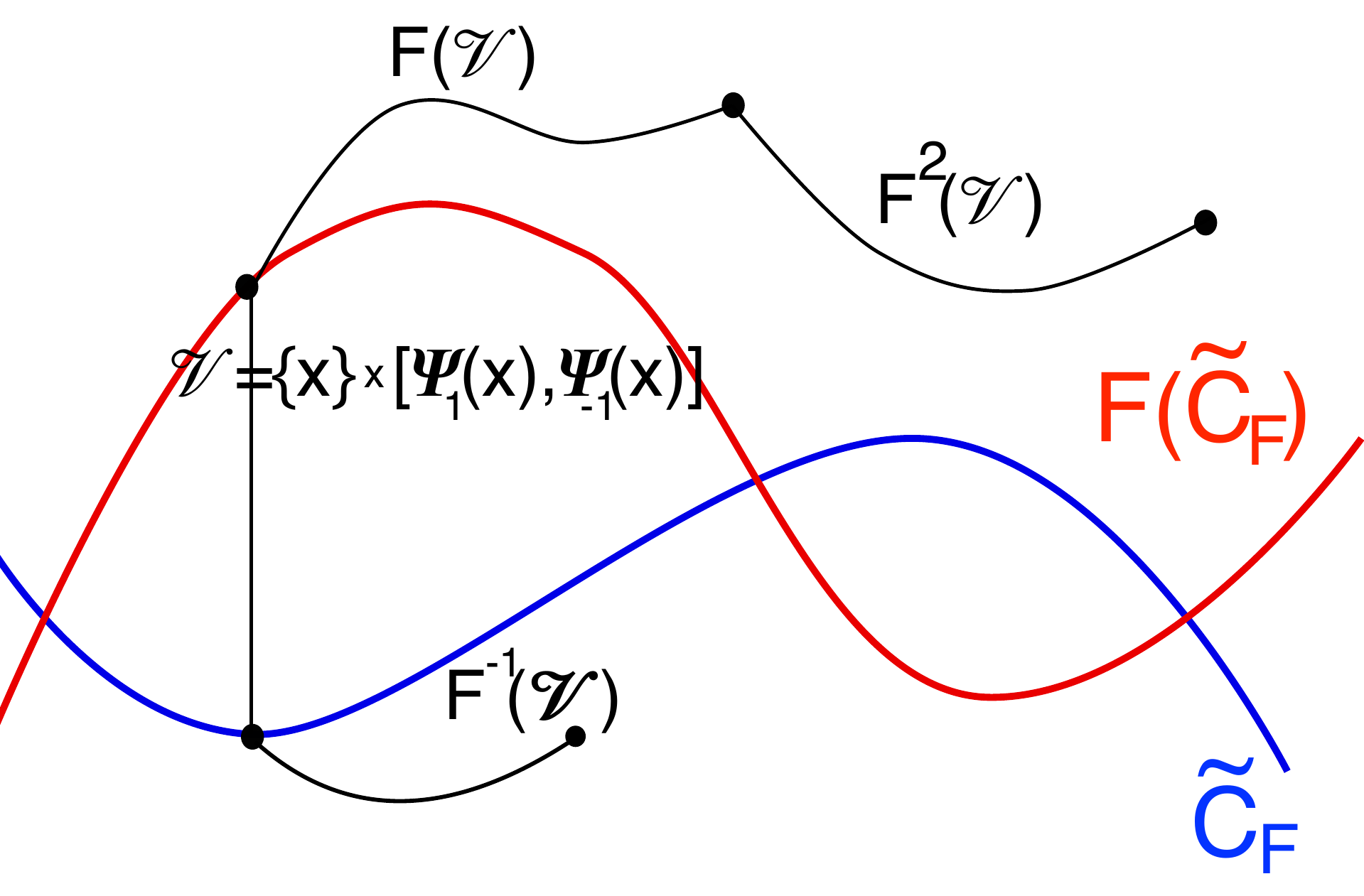}
		\caption{If $\Psi_1(x)<\Psi_{-1}(x)$, then, for every $n\geq 1$, the image by $F^n$ of $\{x\}\times[\Psi_1(x),\Psi_{-1}(x)]$ projects injectively on the horizontal axis.}
		\label{dessinFLC}
	\end{figure}

\begin{notazione}\label{notazione X+ e X-}
	Denote
	\begin{equation}\label{def X-}
	X^-:=\left\{ (x,y)\in\R^2 :\ \Psi_1(x)< y<\Psi_{-1}(x) \right\},
	\end{equation}
	\begin{equation}\label{def X+}
	X^+:=\left\{ (x,y)\in\R^2 :\ \Psi_{-1}(x)< y<\Psi_1(x) \right\}.
	\end{equation}
\end{notazione}
\begin{remark}\label{remark mathscr X}
	Thanks to Lemma \ref{lemma sequences}, if $z$ is not fixed by $f$, then:\vspace{3pt}
	
	\begin{itemize}
		\item[$(1)$] either the sequence $(p_1\circ F^n(z))_{n\in\Z}$ is strictly monotone;\vspace{3pt}
		
		\item[$(2)$] or there exists $n_0\in\Z$ such that the sequences $(p_1\circ F^n(z))_{n\leq n_0}$ and $(p_1\circ F^n(z))_{n\geq n_0}$ are strictly monotone, with different monotonicity;\vspace{3pt}
		
		\item[$(3)$] or there exists $n_0\in\Z$ such that the sequences $(p_1\circ F^n(z))_{n\leq n_0}$ and $(p_1\circ F^n(z))_{n\geq n_0+1}$ are strictly monotone, with different monotonicity, while $p_1\circ F^{n_0}(z)=p_1\circ F^{n_0+1}(z)$.\vspace{3pt}
		
	\end{itemize}
	 Note that the set of points satisfying the second item is nothing but the union of $\ \bigcup_{n\in\Z}F^n(X^-)$ and $\bigcup_{n\in\Z}F^n(X^+)$ and that these two sets are disjoint because they correspond to complementary monotonicity. Note also that the sets $X^-$ and $X^+$ are wandering open sets: the sets $F^n(X^-), n\in\Z,$ are pairwise disjoint, as are the sets $F^n(X^+), n\in\Z$.
	 
	 Let us write $\tilde{C}_F=\tilde{C}_F^+\sqcup \tilde{C}_F^-\sqcup \text{Fix}(f)$, where $(x,\Phi_1(x))\in\tilde{C}_F^+$ if $\Psi_{-1}(x)>\Psi_1(x)$ and $(x,\Psi_1(x))\in\tilde{C}_F^-$ if $\Psi_{-1}(x)<\Psi_1(x)$. The set of points satisfying the third item is the (disjoint) union of $\bigcup_{n\in\Z}F^n(\tilde{C}_F^+)$ and $\bigcup_{n\in\Z}F^n(\tilde{C}_F^-)$. Note also that the last two sets are disjoint from $\bigcup_{n\in\Z}F^n(X^+)$ and $\bigcup_{n\in\Z}F^n(X^-)$.
\end{remark}
\begin{lemma}\label{lemma either or}
	Let $f:\A\rightarrow\A$ be a positive twist map with no point with conjugate points and $F:\R^2\rightarrow\R^2$ be a lift of $f$. Let $x_1<x_2$ and suppose that $\Psi_1(x_1)<\Psi_{-1}(x_1)$ and $\Psi_{-1}(x_2)<\Psi_1(x_2)$. Then\vspace{5pt}
	
	\begin{itemize}
		\item[$(a)$] for every $y\in[\Psi_1(x_1),\Psi_{-1}(x_1)]$ and for every $n\in\N$ it holds that  $p_1\circ F^n(x_1,y)<x_2$
	\end{itemize}
\indent or 
\begin{itemize}
		\item[$(b)$] for every $y\in[\Psi_{-1}(x_2),\Psi_1(x_2)]$ and for every $n\in\N$ it holds that  $p_1\circ F^{-n}(x_2,y)>x_1$.
		\end{itemize}
\end{lemma}
\begin{proof}
	As explained in the proof of Lemma \ref{lemma sequences}, the union of the images by $F^n, n\geq 1,$ of the vertical segment $\{x_1\}\times[\Psi_1(x_1),\Psi_{-1}(x_1)]$ is the graph of a continuous function $\chi_1$ defined on an interval $[x_1,x_1')$, where $x_1'\in(x_1,+\infty]$. This graph is strictly above the graph of $\Psi_0=\max(\Psi_1,\Psi_{-1})$ except at its left end. 
	\noindent Similarly, the union of the images by $F^{-n},n\geq 1,$ of the vertical segment $\{x_2\}\times[\Psi_{-1}(x_2),\Psi_1(x_2)]$ is the graph of a continuous function $\chi_2$ defined on an interval $(x_2',x_2]$, where $x_2'\in[-\infty,x_2)$, and this graph is strictly above the graph of $\Psi_0$ except at its right end. 
	These graphs are contained in $\bigcup_{n\in\Z}F^n(X^-\cup\tilde{C}_F^-)$ and $\bigcup_{n\in\Z}F^n(X^+\cup\tilde{C}_F^+)$ respectively and so are disjoint. By the Intermediate Value Theorem applied at $\chi_1-\chi_2$, one deduces that $x_1'\leq x_2$ or $x_1\leq x_2'$.\\
\end{proof}
\begin{proposizione}\label{prop graph rational}
	Let $f:\A\rightarrow\A$ be an exact symplectic positive twist map with no point having conjugate points and $F:\R^2\rightarrow\R^2$ be a lift of $f$. Then it holds that $	\tilde{C}_F=\{z\in\R^2 :\ F(z)=z\}$.
\end{proposizione}
\begin{proof}
	By Lemma \ref{lemma graph Cg}, the set $\tilde{C}_{F}=\{ (x,y)\in\R^2 :\ p_1\circ F(x,y)=x \}$	is the graph of a function $\Psi_1:\R\rightarrow\R$ that lifts a function $\psi_1:\T\rightarrow\R$ of class $\mathcal{C}^1$, while $F(\tilde{C}_F)$ is the graph of a function $\Psi_{-1}:\R\rightarrow\R$ that lifts a function $\psi_{-1}:\T\rightarrow\R$ of class $\mathcal{C}^1$. We want to prove that $\Psi_1=\Psi_{-1}$ to get the proposition. We will argue by contradiction supposing that $\Psi_1\neq\Psi_{-1}$. We know that $\int_{\T}\psi_{-1}(x)-\psi_1(x)\ dx=0$ because $f$ is exact symplectic. So there exists $x_1<x_1'<x_2<x_2'$ such that $\Psi_{-1}-\Psi_1$ vanishes on $x_1,x_1',x_2,x_2'$, is positive on $(x_1,x_1')$ and negative on $(x_2,x_2')$. Let us define
	$$
	X_1^-:=\{ (x,y)\in(x_1,x_1')\times\R :\ \Psi_1(x)<y<\Psi_{-1}(x) \}
	$$
	and
	$$
	X_2^+:=\{ (x,y)\in(x_2,x_2')\times\R :\ \Psi_{-1}(x)<y<\Psi_1(x) \}.
	$$
By Lemma \ref{lemma either or}, we know that $\bigcup_{n\geq 0}F^n(X_1^-)$ is on the left of the vertical $\{x_2'\}\times\R$ (oriented as $y$ increasing) or $\bigcup_{n\geq 0}F^{-n}(X_2^+)$ is on the right of $\{x_1\}\times\R$. We will suppose that we are in the first case, the other case can be treated similarly. We also know that $\bigcup_{n\geq 0}F^n(X_1^-)$ is on the right of $\{x_1\}\times\R$ and so $\bigcup_{n\geq 1}F^n(X_1^-)$ is below the graph $F(\{x_1\}\times\R)$ because $F$ is a positive twist map. But we have seen in the proof of Lemma \ref{lemma either or} that $\bigcup_{n\geq 1}F^n(X_1^-)$ is above the graph of $\Psi_0=\max(\Psi_1,\Psi_{-1})$. We conclude that $\bigcup_{n\geq 1}F^n(X_1^-)$ is bounded. We have seen in Remark \ref{remark mathscr X} that the $F^n(X_1^-), n\geq 1,$ are pairwise disjoint. The conclusion comes from the fact that $F$ preserves the Lebesgue measure.\\
\end{proof}

~\newline
\noindent\textit{Proof of Theorem \ref{teo intro 2}.} Theorem \ref{teo intro 2} can be deduced easily from Proposition \ref{prop graph rational} and from classical results about twist maps that we will recall and that can be found, for example, in \cite{Her83}. Let $f:\A\rightarrow\A$ be an exact symplectic positive twist map. We fix a lift $F:\R^2\rightarrow\R^2$ of $f$ and will refer to this lift, while talking about rotation numbers. We endow the space $\mathcal{C}^0(\T,\R)$ of continuous maps $\psi:\T\rightarrow\R$ with the uniform topology. We denote $\mathcal{G}(f)$ the set of functions $\psi\in\mathcal{C}^0(\T,\R)$ such that the graph of $\psi$ is invariant and define $\rho(\psi)$ as being the rotation number of the graph of $\psi$.\vspace{3pt}

\begin{itemize}
	\item[$(i)$] For every $\psi,\psi'\in\mathcal{G}(f)$, we have $\rho(\phi)=\rho(\phi')$ if $\psi-\psi'$ vanishes.\vspace{3pt}
	
	\item[$(ii)$] For every $\psi,\psi'\in\mathcal{G}(f)$, we have $\rho(\psi)<\rho(\psi')$ if $\psi<\psi'$.\vspace{3pt}
	
	\item[$(iii)$] There exists at most one map $\psi\in\mathcal{G}(f)$ such that $\rho(\psi)=\rho$, if $\rho\in\R\setminus\mathbb{Q}$.\vspace{3pt}
	
	\item[$(iv)$] If there exists $\psi\in\mathcal{G}(f)$ such that $\rho(\psi)=\rho\in\mathbb{Q}$ and such that its graph contains only periodic points, there is no map $\psi'\neq\psi$ in $\mathcal{G}(f)$ such that $\rho(\psi')=\rho$.\vspace{3pt}
	
	\item[$(v)$]The map $\psi\mapsto\rho(\psi)$ is continuous in $\mathcal{G}(f)$.\vspace{3pt}
	
	\item[$(vi)$] By Birkhoff's theory, for every compact set $K\subset\A$, there exists $M_K>0$ such that $\psi$ is $M_K$-Lipschitz if $\psi$ belongs to $\mathcal{G}(f)$ and if its graph is included in $K$.\vspace{3pt}
	
\end{itemize}
Suppose now that $f$ has no point with conjugate points. For every $\rho\in\mathbb{Q}$, define $G_{\rho}=F^q-(p,0)$, where $\rho=\frac{p}{q},p\in\Z,q\in\N^*$ is written in an irreducible way. Applying Proposition \ref{prop graph rational} to $G_{\rho}$, one deduces that $\text{Fix}(G_{\rho})$ is the graph of a continuous map $\psi_{\rho}:\T\rightarrow\R$. This graph is invariant by $f$ and it is the set of periodic points of $f$ of rotation number $\rho$. A periodic point is above this graph if its rotation number is larger that $\rho$ and below if it is smaller. In particular one has $\rho<\rho'\Leftrightarrow\psi_{\rho}<\psi_{\rho'}$.

Fix $n\geq 0$. The maps $\psi_{\rho},\vert\rho\vert\leq n$, are uniformly bounded, but they are also uniformly Lipschitz by $(vi)$. So, by Ascoli-Arzelà's theorem the set $\mathcal{F}_n:=\{\psi_{\rho} :\ \rho\in\mathbb{Q}\cap[-n,n]\}$ is relatively compact in $\mathcal{C}^0(\T,\R)$. Moreover, it holds that $\overline {\mathcal{F}}_n\subset \mathcal{G}(f)$ and that the image of $\overline{\mathcal{F}}_n$ by $\rho$, being compact, is equal to $[-n,n]$. One knows by $(iii)$ and $(iv)$ that $\rho$ is injective on $\overline{\mathcal{F}}_n$ and consequently induces a homeomorphism from $\overline{\mathcal{F}}_n$ to $[-n,n]$. We note $\rho\mapsto\psi_{\rho}$ its inverse. The set $\mathcal{F}:=\{ \psi_{\rho} :\ \rho\in\mathbb{Q} \}$ can be written as $\mathcal{F}=\bigcup_{n\geq 1}\mathcal{F}_n$. This implies that our original family $(\psi_{\rho})_{\rho\in\mathbb{Q}}$ can be extended to a continuous and injective family $(\psi_{\rho})_{\rho\in\R}$. The function $z\mapsto \rho(z)$ is locally constant on $\text{Fix}(f)$ and so bounded on every compact set of $\A\cap\text{Fix}(f)$. This implies that $\lim_{n\rightarrow+\infty}	\psi_n(x)=+\infty$ and more generally that $\lim_{\rho\rightarrow+\infty}\min\psi_{\rho}(x)=+\infty$ because $\psi_{\lfloor \rho\rfloor}<\psi_{\rho}$. Similarly, it holds that $\lim_{\rho\rightarrow-\infty}\max\psi_{\rho}(x)=-\infty$. This means that the map $\rho\mapsto \psi_{\rho}$ is proper and so defines a homeomorphism from $\R$ to $\overline{\mathcal{F}}=\bigcup_{n\geq 1}\overline{\mathcal{F}}_n$. Let us fix $x\in\T$. The map $\rho\mapsto\psi_{\rho}(x)$ is continuous, increasing and proper. So it defines a homeomorphism from $\R$ to $\R$. More precisely, the map $\A\ni(x,\rho)\mapsto (x,\psi_{\rho}(x))\in\A$ is a homeomorphism, which means that $f$ is $\mathcal{C}^0$-integrable.\\
\hfill\qed

\begin{remark}\label{adapt bounded annulus}
	The proof of Theorem \ref{teo intro 2} can be adapted to any $f$-invariant bounded essential subannulus. Let $\psi:\T\rightarrow\R$ be an $f$-invariant essential curve and let $U$ be the essential subannulus that is the connected component of $\A\setminus\psi(\T)$ lying above the curve. Let $\bar{\rho}$ be the rotation number of $\psi$. 
	We can consider the set $\mathcal{F}$ of functions $\psi_{\rho}$ only for $\rho\in\mathbb{Q},\rho>\bar{\rho}$. Then the exact symplectic twist map restricted to $U$ $f_{\vert U}:U\rightarrow U$ is $\mathcal{C}^0$-integrable if and only if $f_{\vert U}$ does not have points with conjugate points. The adapted argument holds for the connected component lying under the $f$-invariant curve. Thus, the result holds for any $f$-invariant bounded essential subannulus by considering it as intersection of a lower and an upper $f$-invariant unbounded essential subannulus.
\end{remark}

In the end, we want to present an alternative proof for the result in \cite{ChengSun} that is due to an anonymous referee. His/her proof uses the results in \cite{bangert} and so a different approach from our one. In Proposition 1 in \cite{ChengSun}, Cheng and Sun prove that if an area preserving twist map has no points with conjugate points, then the Hessian matrix of the second partial derivatives of the associated action functional is positively definite. Using Bangert's notations (see \cite{bangert}), the referee observes that for each rational $\frac{p}{q}$, it holds that $p_0(\mathcal{M}^{per}_{p/q})=\R$\footnote{Here $\mathcal{M}^{per}_{p/q}$ corresponds to the set of periodic orbits that minimize the action functional and have rotation number equal to $p/q$.}. Otherwise, there exist neighbouring elements $\textbf{x}^-<\textbf{x}^+$ in $\mathscr{M}^{per}_{p/q}$ and, by the Mountain Pass Lemma, we find a minmax $(p,q)$-periodic $\textbf{x}^*$ between $\textbf{x}^-$ and $\textbf{x}^+$: this would contradict Proposition 1 in \cite{ChengSun}. Repeating then the same arguments used in the proof of Theorem \ref{teo intro 2}, one can conclude.

~\newline
~\newline
\textbf{Acknowledgements.} \noindent The first author is very grateful to prof. Marie-Claude Arnaud and Andrea Venturelli for precious advices and discussions. The authors thank the anonymous referees for their suggestions and comments which helped to improve the presentation of the paper.

\bibliographystyle{alpha}
\bibliography{Bibliography}
\end{document}